\documentclass[letterpaper, 10 pt, conference]{ieeeconf}  

\IEEEoverridecommandlockouts                              

\usepackage{amsmath,graphicx}
\usepackage{amssymb}
\usepackage{amsfonts}
\usepackage{physics}

\newcommand{\hide}[1]{}

\usepackage{algorithm} 
\usepackage{algpseudocode} 

\usepackage{subcaption}
\usepackage{caption}
\usepackage{xcolor}

\title{\LARGE \bf
Distributed optimization on directed graphs based on inexact ADMM with partial participation
}

\author{Dingran Yi and Nikolaos M. Freris\thanks{School of Computer Science, University of Science and Technology of China, Hefei, Anhui, 230027, China. 
Emails: \texttt{ydr0826@mail.ustc.edu.cn, nfr@ustc.edu.cn}.}}

\begin{document}

\maketitle
\thispagestyle{empty}
\pagestyle{empty}

\begin{abstract}

We consider the problem of minimizing the sum of cost functions pertaining to agents over a network whose topology is captured 
by a directed graph (i.e., asymmetric communication). We cast the problem into the ADMM setting, via a consensus constraint, 
for which both primal subproblems are solved inexactly. In specific, the computationally demanding local minimization step is replaced by a single gradient step, while the averaging step 
is approximated in a distributed fashion. Furthermore, partial participation is allowed in the implementation of the algorithm.  Under 
standard assumptions on strong convexity and Lipschitz continuous gradients, we establish linear convergence and characterize the rate in terms of the connectivity of the graph and the conditioning of the problem. Our line of analysis provides a sharper convergence rate compared to Push-DIGing. Numerical experiments corroborate the merits of the proposed solution in terms of superior rate as well as computation and communication savings over baselines.

\end{abstract}

\section{INTRODUCTION}

Distributed multi-agent optimization has found paramount applications across various fields such as in control \cite{nedic2018distributed}, signal processing \cite{dimakis2010gossip,sopasakis2016accelerated}, machine learning and data mining \cite{verbraeken2020survey,vlachos2015compressive}, and wireless sensor networks \cite{sensor,intanagonwiwat2003directed}. The archetypal problem is to 
\begin{equation}
    \underset{\hat{x}\in\mathbb{R}^d}{\text{minimize}}\qquad f\left(\hat{x}\right)=\sum_{i=1}^n f_i(\hat{x}) , \label{prob1}
\end{equation} 
where $f_i(\cdot)$ is a local cost function corresponding to agent~$i$. Distributed optimization amounts to solving (\ref{prob1}) over the common decision vector $\hat{x}$ by a synergy of local computations and 
communication exchanges. In specific, agent $i$ holds a local variable $x_i$ which is updated based on its
local cost $f_i(\cdot)$ alongside information obtained by its neighbors 
(e.g., their local variables or gradients), 
and \emph{consensus} (i.e., $x_1=\cdots=x_n$) is achieved 
asymptotically.

There has been extensive work on the subject, especially for the case that the communication network topology is symmetric (i.e., 
an undirected graph) \cite{subgrad,makhdoumi2017convergence,chang2014multi,7942055}. Nevertheless, it is quite common to have unidirectional communication links in wireless networks due to heterogeneity in 
transceivers or perceived levels of interference \cite{intanagonwiwat2003directed}. Algorithms on directed graphs can be roughly classified into: a) primal 
methods and b) primal-dual methods, most notably based on the \emph{Alternating Direction Method of Multipliers} 
(ADMM) \cite{boyd2011distributed}. In the former case, a local gradient step is used along with weighted averaging across neighboring agents. A sublinear convergence is established \cite{nedic2014distributed,nedic2016stochastic} even for strongly convex problems, while linear convergence can be retrieved using gradient
tracking \cite{nedic2017achieving}.

Recent work has developed ADMM-based methods for distributed optimization on directed graphs \cite{rokade2020distributed,jiang2021asynchronous,khatana2020dc}. In specific, 
\cite{rokade2020distributed} uses dynamically updated weights for local averaging and establishes linear convergence under strong convexity. The authors in \cite{jiang2021asynchronous} adopted $\epsilon-$inexact consensus and proposed an asynchronous method that requires a finite number of communication steps per round. \cite{khatana2020dc} allows for both equality and inequality constraints and establishes either sublinear rate to the exact solution or linear 
rate to a neighborhood of optimality.

Notwithstanding, the prior art based on ADMM requires an exact solution of the local subproblems, which may incur heavy 
computational burden unsuitable for resource-constrained devices. Besides, it is not amenable to partial agent participation, an 
imperative requirement in real scenarios where user unavailability is common (due to variable operating conditions such as battery level
and network bandwidth) and synchronization is difficult \cite{freris2010fundamental} . 
In order to address these challenges, we propose \emph{IPD (Inexact, Partial participation, Directed graph)}.\\
\textbf{Contributions}:
\begin{enumerate}
\item We propose a primal-dual method for distributed optimization on directed graphs that alleviates the computational load by inexact solution of the local subproblems (using a single gradient step).
\item Under standard assumptions, we establish linear convergence and characterize the rate with respect to the graph connectivity and the conditioning of the problem (Thm.~1 and Cor.~1).
\item The method allows for partial user participation at each iteration of the algorithm. Thm.~2 establishes linear convergence with high probability and reveals its dependency on the activation probability.
\item Our analysis provides a sharper characterization of the rate compared to the state-of-the-art Push-DIGing method with which it shares comparable  communication and computation costs.
\item Experiments on two real-life machine learning datasets demonstrate merits in terms of a) faster rate compared to Push-DIGing, b) computation and communication savings over baselines.
\end{enumerate}

\section{Notation}
\label{sec2}
The network topology is captured by a directed graph $\mathcal{G}=\left\{\mathcal{V},\mathcal{E}\right\}$, where $\mathcal{V}$ is the set of agents (with cardinality $n:=\lvert \mathcal{V}\rvert$) and $\mathcal{E}$ is the set of directed communication links: $\left(i,j\right) \in \mathcal{E}$ if and only if agent $i$ can send a message to agent $j$. We define the set of agent $i$'s in-neighbors as $\mathcal{N}_i^{\text{in}}:=\left\{j:\left(j,i\right) \in \mathcal{E}\right\}$, and its out-degree by $d_i:= |\left\{j: (i,j)\in\mathcal{E}\right\}|$. The maximum out-degree is denoted by $d_{\max}:=\max_{i\in\mathcal{V}} d_i$, while  $D:=\text{diag}(d_1,d_2,\hdots, d_n)$ is a diagonal 
matrix with entries the out-degrees of the corresponding agents. The adjacency matrix $A\in\mathbb{R}^{n\times n}$ satisfies $A_{ij}=1$ if $(j,i) \in \mathcal{E}$ and $A_{ij}=0$ otherwise. We further define $P:=\left(I+AD^{-1}\right)/2$ and for $\lambda_2(P)$ its second largest eigenvalue, while we use $\phi$ for the diameter of the graph. All vectors are meant as column vectors. The notation $x_i\in\mathbb{R}^d$ is for the local vector of agent $i$, while $x\in
\mathbb{R}^{nd}$ is reserved for the concatenation, i.e.,  $x^T:=[x_1^T,\hdots,x_n^T]^T$ (and analogously for other variables). We let  $\bar{x}^k:=\frac{1}{n}\sum\limits_{i=1}^n x_i^k\otimes \textbf{1}_n$, where $\mathbf{1}$ means the all-one vector, and $x^k_\perp:=x^k-\bar{x}^k$ where superscript $k$ corresponds to the $k-$th iterates; analogous definitions apply for $\bar{z}^k$ and $z^k_\perp$. Additionally, we define $F\left(x\right):=\sum\limits_{i=1}^n f_i\left(x_i\right)$ and  the consensus set $\mathcal{C}:=\left\{x:x_1=\cdots=x_n\right\}$ with corresponding indicator function:
\begin{gather*}
    I_{\mathcal{C}}\left(x\right)=
    \begin{cases}
    0, \quad x \in \mathcal{C}\\
    \infty, \quad \text{else}
    \end{cases}
\end{gather*}
Finally, we let $[n]:=\{1,\hdots,n\}$ for $n\in\mathbb{N}$. 



\section{Proposed Method}
\label{sec3}
To cast problem (\ref{prob1}) into the setting of ADMM, we re-write is as:
\begin{eqnarray}
    \underset{x,z\in\mathbb{R}^{nd}}{\text{minimize}}&\ &F\left
(x\right)+I_{\mathcal{C}}\left(z\right) \notag \\
    \text{subject to} &\ &x=z \label{prob2}
\end{eqnarray}
The augmented Lagrangian (AL) for (\ref{prob2}) is given by:
\begin{gather*}
    L_{\rho}\left(x,z,y\right)=f\left(x\right)+I_{\mathcal{C}}\left(z\right)+y^T\left(x-z\right)+\frac{\rho}{2}\norm{x-z}^2,
\end{gather*}
where $y\in\mathbb{R}^{nd}$ is the dual variable and $\rho>0$. The iterations of ADMM are given by sequential alternating minimization of the AL over $x,z$ plus a dual ascent step, as follows: 
\begin{subequations}
\begin{align}
    x^{k+1}&=\underset{x}{\text{argmin}}\,\,\, L_{\rho}\left(x,z^k,y^k\right),\label{exactx}\\
    z^{k+1}&=\underset{z}{\text{argmin}}\,\,\, L_{\rho}\left(x^{k+1},z,y^k\right),\label{exactz}\\
    y^{k+1}&=y^k+\rho\left(x^{k+1}-z^{k+1}\right).\label{updatea}
\end{align}
\end{subequations}
\begin{algorithm}[t]
	\caption{IPD} 
	\textbf{Initialization}: $x_i^0=z_i^0=y_i^0=0,w_i^0(0) \in \left(0, d^{-(2\phi+1)}_{\max}\right]$
	\begin{algorithmic}[1]
		\For {$k=0,1,\hdots$}
		\hide{\State Agent $i\in [m]$ is active with probability $p_i>0$.}
		\For {each active agent $i$}
		\State Compute $x_i^{k+1}$ using (\ref{updatex})
		\State Initialize $\xi_i^{k+1}(0)=x_i^{k+1}$
		\For {$b=0,1,\hdots,B-1$}
		\State Broadcast $w_i^k(b)$ and $\xi_i^{k+1}(b)$
		\State Compute $w_i^k(b+1)$ using (\ref{updatexi})
		\State Compute $\xi_i^{k+1}(b+1)$ using (\ref{updatez})
		\EndFor
		\State Set $w_i^{k+1}(0)=w_i^k(B)$ and $z_i^{k+1}=\xi_i^{k+1}(B)$
		\State Compute $y_i^{k+1}=y_i^k+\rho\left(x_i^{k+1}-z_i^{k+1}\right)$
		\EndFor
		\EndFor 
	\end{algorithmic} 
\end{algorithm}Step (\ref{exactx}) decomposes to local optimization subproblems at the agents. Solving these exactly is generally computationally burdensome, therefore \emph{inexact} minimization is invoked in the form of a single gradient descent step (with $\eta>0$) as:  
\begin{equation} \label{updatex}
    x_i^{k+1}=x_i^k-\eta\left(\nabla f_i\left(x_i^k\right)+y_i^k+\rho\left(x_i^k-z_i^k\right)\right).
\end{equation}
Problem (\ref{exactz}) is a quadratic program with closed form solution given by: 
\begin{gather*}
    z_i^{k+1}=\frac{1}{n}\sum\limits_{j=1}^n\left(x_j^{k+1}+\frac{y_j^k}{\rho}\right).
\end{gather*}
The challenge here is that this is a global averaging step that can not be computed in a distributed manner without (prohibitively) extensive message-passing. For this reason, we also opt to solve (\ref{exactz}) \emph{inexactly} by $B$ distributed averaging steps with weights obtained by the weight balancing method proposed in \cite{rokade2020distributed}.
The latter requires each agent to initialize its local weight $w_i^0\left(0\right) \in \left(0,d_{\max}^{-\left(2\phi+1\right)}\right]$ and update it (for $k\geq 0, b\in\{0,\hdots,B-1\}$) using:
\begin{equation}
        w_i^k(b+1)=\frac{1}{2}\left(w_i^k(b)+\frac{1}{d_i}\sum \limits_{j\in \mathcal{N}_i^{\text{in}}}w_j^k(b)\right),\label{updatexi}
\end{equation}
setting $w_i^{k+1}(0)=w_i^k(B)$.\\
We use $\xi_i^k\left(\cdot\right)$ as the proxy for $z_i^k$, which is initialized as $\xi_i^{k+1}(0)=x_i^{k+1}$ and updated ($b \in \{0,\hdots,B-1\}$) using:
\begin{equation}
    \xi_i^{k+1}(b+1)=\left(1-d_i w_i^k(b)\right)\xi_i^{k+1}(b)+\sum\limits_{j \in \mathcal{N}_i^{\text{in}}}w_j^k(b)\xi_j^{k+1}(b),\label{updatez}
\end{equation}
whence we let $z_i^{k+1}:=\xi_i^{k+1}(B)$.\\
These steps can be carried in a distributed fashion using information obtained from in-neighbors which, in turn, suggests that \emph{broadcasting} to out-neighbors suffices for communication. Our method is termed IPD (IPD: inexact, partial participation, directed graph) and is presented as Alg.~1. It supports partial participation (step 2); this is analyzed as random activation with probabilities $q_i$ (Thm.~2). Besides, local computation amounts to an economical single gradient step (step~3). Communication is carried by broadcasting $\xi_i$ and weight $w_i$ (step~6; total cost is $d+1$) which are updated by distributed 
averaging (steps 7-8), while step 11 is for dual ascent. In view of partial participation, an implicit assumption is that information received by broadcasting is stored in a buffer regardless of 
whether the agent is active or not and an active agent is available to carry all operations in steps 5-9.\\



\section{Convergence analysis}
\label{sec4}
The analysis is carried under the following two assumptions:
\newtheorem{assumption}{Assumption} 
\begin{assumption}\label{connected}
The directed graph $\mathcal{G}$ is strongly connected.
\end{assumption}
\begin{assumption}\label{aboutf}
The objective functions $f_i$ satisfy:
\begin{enumerate}
    \item Each $f_i, i\in [n]$ is strongly-convex, i.e., there exists $m_f>0$, such that \,$\forall x,y \in \mathbb{R}^n,\\\left(\nabla f_i\left(x\right)-\nabla f_i\left(y\right)\right)^T\left(x-y\right)\geq m_f \norm{x-y}^2$.
    \item The gradient of each $f_i, i\in [n]$ is Lipschitz continuous with constant $M_f>0$, i.e., $\forall x,y \in \mathbb{R}^n,\\\norm{\nabla f_i\left(x\right)-\nabla f_i\left(y\right)} \leq M_f\norm{x-y}$.
    \item The set of minimizers of $f$ is nonempty.
\end{enumerate}
\end{assumption}
We denote the primal-dual optimal solution of (\ref{prob2}) by $\left(x^\star,z^\star,y^\star\right)$; strong convexity guarantees unique primal solutions, while (\ref{KKTdual}) of the following KKT conditions guarantees uniqueness of the dual optimal solution:
\begin{subequations}
\begin{align}
x_1^{\star}=\cdots=x_n^{\star},\\
x^\star=z^\star,\\
\nabla F\left(x^{\star}\right)+y^{\star}=0.\label{KKTdual}
\end{align}
\end{subequations}
The following lemma establishes an error bound on the deviation of $z^k$ from the average $\bar{z}^k$ that is key in 
establishing our convergence theorem.
\newtheorem{lemma}{Lemma}
\begin{lemma}
Let $x^k_\perp:=x^k-\bar{x}^k,z^k_\perp:=z^k-\bar{z}^k$. Under Assumption \ref{connected}, there exist $\bar{k}>0,\delta \in \left(0,1\right)$, such that the sequence generated by Alg.~1 satisfies$$\norm{z^k_\perp}^2 \leq \delta^{2B-1} \norm{x^k_\perp}^2+\Delta^k\left(\norm{x^k-x^\star}^2+\norm{x^\star}^2\right),$$  $\forall k \geq \bar{k}$, and $\Delta^k \rightarrow 0$ geometrically with rate $\lambda_2(P)$.
\end{lemma}
\begin{proof}
See Appendix.
\end{proof}
We proceed to analyze the convergence of Alg. 1 in two steps: a) Thm.~1 establishes linear convergence with full participation; b) Thm.~2 considers the case of partial participation and establishes the convergence with high probability.
\newtheorem{theorem}{Theorem}
\begin{theorem}[Full Participation]
Under Assumptions \ref{connected} and \ref{aboutf},  by choosing
$\eta\leq\frac{4}{15(M_f+m_f)}, \rho \in \left(0,\frac{4}{87}\frac{M_f m_f}{M_f+m_f}\right),$
\begin{footnotesize}
\centerline{$B\geq\max\left\{1,\frac{1}{2}\left((\ln \frac{5}{36}/\ln \delta)+1\right),\frac{1}{2}\left(\ln \frac{8M_fm_f}{9(M_f+m_f)^2}/\ln \delta+1\right)\right\},$}
\end{footnotesize}
if all agents are active in each iteration, the sequence generated by Alg.~1 satisfies\\
$$\norm{x^k-x^{\star}}^2+\norm{y^k-y^{\star}}^2=\mathcal{O}\left(\lambda^k\right),$$where\\
\centerline{$\lambda=\max\left\{\frac{2\mu_1}{1+\mu_1},\lambda_2(P)\right\},$}
    $\mu_1=\max\left\{\frac{c_1+c_2}{2c_2},1-\frac{2}{3}\rho c_3\right\},$\\
    $c_1:=\frac{1}{2\eta}+2\rho+3\rho\delta^{2B-1}-\frac{m_f M_f}{m_f+M_f},$\\
    $c_2:=\frac{1}{2\eta}-\delta^{2B-1}\left[\frac{5\rho}{4}+\frac{1}{\eta}+\frac{3(m_f+M_f)}{4}\right]-\rho,$\\
    $c_3:=\min\left\{\frac{1}{3(m_f+M_f)},\eta^2\left(\frac{3(m_f+M_f)}{16}-\rho\right),\frac{c_2-c_1}{\rho^2(4+4\delta^{2B-1})}\right\}.$
\end{theorem}
\begin{proof}
See Appendix.
\end{proof}
The following corollary shows that for appropriate parameter choice the convergence rate depends on a) the condition number $\kappa$ (with a dependency of $\frac{1}{\kappa}$ that is reminiscent of first order methods without acceleration) and b) the connectivity of the graph ($\lambda_2(P)$).
The lower bound on B increases logarithmically with the condition number and decreases logarithmically with the parameter $\delta$ (pertaining to the graph topology).
\newtheorem{corollary}{Corollary}
\begin{corollary}
Let $\eta=\frac{4}{15\left(M_f+m_f\right)}$, $\rho = \frac{2}{87}\frac{M_f m_f}{M_f+m_f}$, and $B> \frac{1}{2}+\frac{\ln (1000(\kappa +1))}{2 \ln 1/\delta}$, and define $\kappa:=\frac{M_f}{m_f}$. Then Alg.~1 with full participation converges linearly with rate 
$$\lambda=\max\left\{1-\mathcal{O}\left(\frac{1}{\kappa}\right), \lambda_2(P)\right\}.$$
\end{corollary}
In the following theorem, parameters $\eta$ (stepsize), $\rho$ (penalty coefficient), $B$ (inner-loop rounds) are as in Thm.~1. Convergence with partial participation is established under a simple stochastic model that assumes agents are activated with probability $q_i>0$ at each iteration (step 2 of Alg. 1).
\begin{theorem}[Partial Participation]
Let Assumptions \ref{connected} and \ref{aboutf} and parameters $\eta$, $\rho$, $B$ as in Thm.~1. If at every round agent $i$ is active with probability $q_i$ (i.i.d. across rounds) with $q_{\min}:=\min_{i\in[n]}q_i>0$, then for any $\epsilon \in (0,1)$, with probability at least $1-\epsilon$, the following holds: $$\norm{v^k-v^\star}^2=\mathcal{O}\left(\lambda_p^k\right),$$ where $\lambda_p$ can be arbitrarily chosen from $\left(\lambda_1, 1\right)$,\\$\lambda_1=\max\left\{\lambda_2(P),1-q_{\min}\frac{1-\mu_1}{1+\mu_1}\right\}$, $v^k:=\left[\left(x^k\right)^T \left(y^k\right)^T\right]^T$ and $v^\star:=\left[\left(x^\star\right)^T \left(y^\star\right)^T\right]^T$.
\end{theorem}
\begin{proof}
See Appendix.
\end{proof}
\newtheorem{remark}{Remark}
\begin{remark}
Convergence in Thm. 2 is established with high probability, where the dependency on $\epsilon$ is hidden in $\mathcal{O}(\cdot)$, since the established rate analysis is asymptotic. In specific, for any $\epsilon \in \left(0,1\right)$, with
probability at least $1-\epsilon$, there exists some $K=K(\epsilon,\lambda_p)$,
such that $\forall k>K, \norm{v^{k}-v^\star}^2 \leq~\lambda_p^k$.
\end{remark}
\section{Comparison with push-diging}
\textbf{Push-DIGing} \cite{nedic2017achieving} is the most popular gradient-based method for directed graphs that achieves linear convergence through gradient tracking (second line in (\ref{s4})). The updates for agent $i$ at iteration $k$ are as follows:
\begin{align}
    u_i(k+1)=&\,\,c_{ii}(u_i(k)-\eta y_i(k))\notag\\&+\sum_{j \in \mathcal{N}_i^{\text{in}}}c_{ij}(k)(u_j(k)-\eta y_j(k)),\label{s1}\\
    v_i(k+1)=&\,\,c_{ii}(k)v_i(k)+\sum_{j \in \mathcal{N}_i^{\text{in}}}c_{ij}(k)v_j(k),\label{s2}\\
    x_i(k+1)=&\,\,u_i(k+1)/v_i(k+1),\label{s3}\\
    y_i(k+1)=&\,\,c_{ii}(k)y_i(k)+\sum_{j \in \mathcal{N}_i^{\text{in}}}c_{ij}(k)y_j(k)\notag\\&+\left\{\nabla f_i(x_i(k+1))-\nabla f_i(x_i(k))\right\}.\label{s4}
\end{align}
where $\eta$ is the stepsize, $C(k)$ is a column-stochastic matrix, and the initialization uses $v_i(0)=1$, $x_i(0)=u_i(0)$, and $y_i(0)=\nabla f_i(x_i(0))$.

We first show that our proposed solution has comparable  computation/communication cost (in fact lower communication cost for the case $B=1$):\\
$\bullet$ for Push-DIGing, in each round the cost for one agent includes one broadcast of $2d+1$, one gradient evaluation, plus a local averaging cost of $d_i^{\text{in}}(2d+1)$ where $d_i^{\text{in}}$ is the in-degree of agent~$i$.\\
$\bullet$ for our method, the cost includes one broadcast of $B(d+1)$, one gradient evaluation, and averaging cost of $d_i^{\text{in}}\left(B(d+1)\right)$.\\
We conclude that for $B=1$, IPD has lower computation and communication costs.

We proceed to compare our rate with the one established in \cite[Theorem 18]{nedic2017achieving}. For Push-DIGing, by denoting $V(k):=\text{diag}\left(v_1(k),\dots,v_n(k)\right)$, it holds that $\sup_k \norm{V(k)^{-1}}_{\max} =\mathcal{O}\left(n^n\right)$ \cite[Equation 49]{nedic2017achieving}, which results in an upper bound selection of stepsize as $\mathcal{O}\left(n^{-n}\right)$; this, in turn, leads to a rate of $1-\mathcal{O}\left(n^{-n}\right)$. This is reminiscent of the analysis technique which was developed to address a more general problem that also considers time-varying graphs. In contrast, the stepsize in our case depends only on the scaling and not on the population (see Thm.~1). We further study the rate experimentally in Fig.~\ref{fig2}, which 
depicts a substantial acceleration for the same stepsize. This is also translated to substantial computation and communication savings (since $B=1$ is used in all comparisons) . The superior rate achieved by our proposed method can further be explained by the weight balancing process in (\ref{updatexi}) which converges to a doubly stochastic matrix. In contrast, (\ref{s4}) in Push-DIGing uses a column-stochastic matrix (i.e., it does not apply the push-sum protocol in the gradient estimation step). 
\section{Experiments}
\label{sec5}
We evaluate IPD on a distributed logistic regression problem:$$f_i(x) := \frac{1}{m_i}\sum_{j=1}^{m_i} \left[\ln\left(1+e^{w_j^Tx}\right) +(1-y_j)w_j^Tx\right],$$ where $m_i$ is the number of data points held by each agent and $\{w_j,y_j\}_{j=1}^{m_i}\subset \mathbb{R}^d\times\{0,1\}$ are labeled samples. We used two datasets from LIBSVM\footnote{Available at https://www.csie.ntu.edu.tw/\textasciitilde cjlin/libsvm/.} and the UCI Machine Learning Repository\footnote{Available at https://archive.ics.uci.edu/ml/index.php.} for Fig.~1 and Fig.~2 respectively. In each case we take 5,000 data points with dimension $d=22$, distribute them uniformly at random across $n=50$ agents. The communication topology is captured by a directed graph obtained by randomly adding edges to the ring graph with probability 0.2. We use the relative cost error as the metric for convergence which is  defined as $\frac{\sum_{i=1}^n f(x_i^k)-f(x_i^\star)}{\sum_{i=1}^n f(x_i^0)-f(x_i^\star)}$.  Fig.~1.a shows the comparison with the alternative of solving the local optimization problem (\ref{exactx}) exactly as in \cite{rokade2020distributed}
: it reveals a large computational saving of 87.5\% in all cases. 
Fig.~1.b illustrates a negligible dependency on the number of communication steps $B$. In fact, all other experiments for both datasets were conducted for $B=1$.
Fig.~2.a shows the effect of increasing participation in the speedup of the algorithm (as expected from Thm.~2). Last but not least, we compared against Push-DIGing in Fig.~2.b for two different stepsizes (common for both methods). The superior convergence rate of IPD translates to 90.4\% computation and 94.9\% communication savings for target accuracy of 0.1.
\begin{figure}[t]
  \centering
  \begin{subfigure}[b]{0.49\linewidth}
    \includegraphics[width=\textwidth]{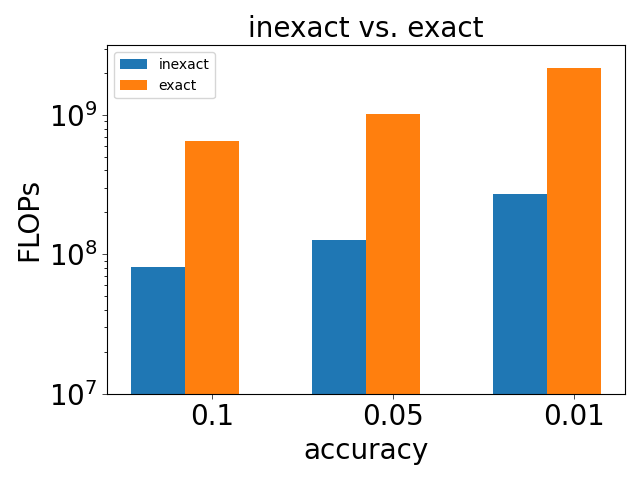}
    \caption{}
  \end{subfigure}
  \begin{subfigure}[b]{0.49\linewidth}
    \includegraphics[width=\textwidth]{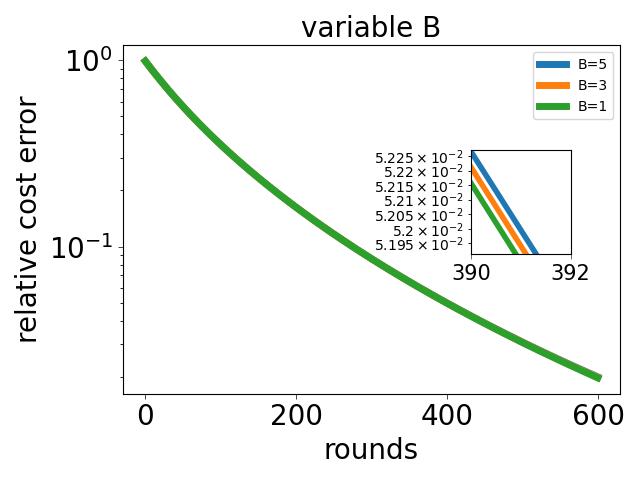}
    \caption{}
  \end{subfigure}\label{vs}
  \caption{Computation cost to reach a target accuracy compared with \cite{rokade2020distributed} (a) and convergence paths for variable $B$ (b) (Full participation is considered in both cases).}
  \label{fig1}
  \begin{subfigure}[b]{0.49\linewidth}
    \includegraphics[width=\textwidth]{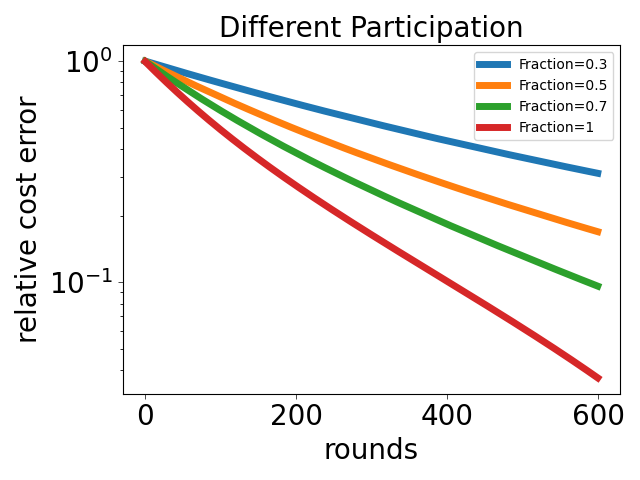}
    \caption{}
  \end{subfigure}
  \begin{subfigure}[b]{0.49\linewidth}
    \includegraphics[width=\textwidth]{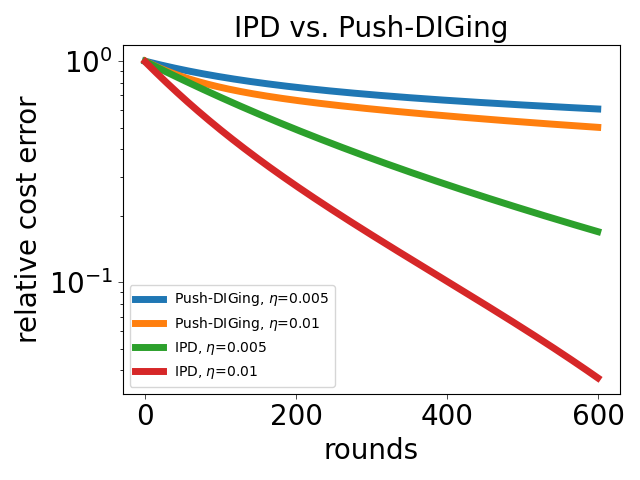}
    \caption{}
  \end{subfigure}
  \caption{Comparison on different levels of participation for $B=5$ (a) and comparison (full participation) with Push-DIGing (b).
}
  \label{fig2}
\end{figure}
\section{Conclusion}
This paper proposed proposed IPD, a primal-dual method for distributed optimization over directed graphs. The two primal subproblems are solved inexactly: one step of gradient descent for the local optimization problem ($x-$variable) and distributed averaging based on weight balancing ($z-$variable). IPD was shown both theoretically and experimentally to have faster convergence than Push-DIGing (Sec. V) as well as substantial computation and communication savings over baseline methods. The established linear rate gives a decomposition in terms of the problem conditioning and the network connectivity (Thm.~1 and Cor.~1), similar with first order methods on undirected graphs. Furthermore, the lower bound for $B$ (number of communications per round) was shown to have a logarithmic dependency with conditioning and connectivity (Cor.~1). A distinctive attribute of IPD is the feasibility of partial agent participation, which is crucial in large-scale real systems (the rate was established in Thm.~2).

\bibliographystyle{IEEE.bst}
\bibliography{ref.bib}

\begin{thebibliography}{10}

\bibitem{nedic2018distributed}
A.~Nedi{\'c} and J.~Liu,
\newblock ``Distributed optimization for control,''
\newblock {\em Annual Review of Control, Robotics, and Autonomous Systems},
  vol. 1, pp. 77--103, 2018.

\bibitem{dimakis2010gossip}
A.~Dimakis, S.~Kar, J.~Moura, M.~Rabbat, and A.~Scaglione,
\newblock ``Gossip algorithms for distributed signal processing,''
\newblock {\em Proceedings of the IEEE}, vol. 98, no. 11, pp. 1847--1864, 2010.

\bibitem{sopasakis2016accelerated}
P.~Sopasakis, N.~M. Freris, and P.~Patrinos,
\newblock ``Accelerated reconstruction of a compressively sampled data
  stream,''
\newblock in {\em {IEEE} European Signal Processing Conference (EUSIPCO)},
  2016, pp. 1078--1082.

\bibitem{verbraeken2020survey}
J.~Verbraeken, M.~Wolting, J.~Katzy, J.~Kloppenburg, T.~Verbelen, and J.~S.
  Rellermeyer,
\newblock ``A survey on distributed machine learning,''
\newblock {\em {ACM} computing surveys (csur)}, vol. 53, no. 2, pp. 1--33,
  2020.

\bibitem{vlachos2015compressive}
M.~Vlachos, N.~M. Freris, and A.~Kyrillidis,
\newblock ``Compressive mining: fast and optimal data mining in the compressed
  domain,''
\newblock {\em The VLDB Journal}, vol. 24, no. 1, pp. 1--24, 2015.

\bibitem{sensor}
N.~M. {Freris}, H.~{Kowshik}, and P.~R. {Kumar},
\newblock ``Fundamentals of large sensor networks: Connectivity, capacity,
  clocks, and computation,''
\newblock {\em Proceedings of the IEEE}, pp. 1828--1846, 2010.

\bibitem{intanagonwiwat2003directed}
C.~Intanagonwiwat, R.~Govindan, D.~Estrin, J.~Heidemann, and F.~Silva,
\newblock ``Directed diffusion for wireless sensor networking,''
\newblock {\em IEEE/ACM {T}ransactions on {N}etworking}, vol. 11, no. 1, pp.
  2--16, 2003.

\bibitem{subgrad}
A.~{Nedic} and A.~{Ozdaglar},
\newblock ``Distributed subgradient methods for multi-agent optimization,''
\newblock {\em IEEE Transactions on Automatic Control}, vol. 54, pp. 48--61,
  2009.

\bibitem{makhdoumi2017convergence}
A.~Makhdoumi and A.~Ozdaglar,
\newblock ``Convergence rate of distributed {ADMM} over networks,''
\newblock {\em IEEE Transactions on Automatic Control}, vol. 62, no. 10, pp.
  5082--5095, 2017.

\bibitem{chang2014multi}
T.~H. Chang, M.~Hong, and X.~Wang,
\newblock ``Multi-agent distributed optimization via inexact consensus
  {ADMM},''
\newblock {\em IEEE Transactions on Signal Processing}, vol. 63, no. 2, pp.
  482--497, 2014.

\bibitem{7942055}
N.~S. Aybat, Z.~Wang, T.~Lin, and S.~Ma,
\newblock ``Distributed linearized alternating direction method of multipliers
  for composite convex consensus optimization,''
\newblock {\em IEEE Transactions on Automatic Control}, vol. 63, no. 1, pp.
  5--20, 2018.

\bibitem{boyd2011distributed}
S.~Boyd, N.~Parikh, E.~Chu, B.~Peleato, and J.~Eckstein,
\newblock ``Distributed optimization and statistical learning via the
  alternating direction method of multipliers,''
\newblock {\em Foundations and Trends{\textregistered} in Machine Learning},
  vol. 3, no. 1, pp. 1--122, 2011.

\bibitem{nedic2014distributed}
A.~Nedi{\'c} and A.~Olshevsky,
\newblock ``Distributed optimization over time-varying directed graphs,''
\newblock {\em IEEE Transactions on Automatic Control}, vol. 60, no. 3, pp.
  601--615, 2014.

\bibitem{nedic2016stochastic}
A.~Nedi{\'c} and A.~Olshevsky,
\newblock ``Stochastic gradient-push for strongly convex functions on
  time-varying directed graphs,''
\newblock {\em IEEE Transactions on Automatic Control}, vol. 61, no. 12, pp.
  3936--3947, 2016.

\bibitem{nedic2017achieving}
A.~Nedi{\'c}, A.~Olshevsky, and W.~Shi,
\newblock ``Achieving geometric convergence for distributed optimization over
  time-varying graphs,''
\newblock {\em SIAM Journal on Optimization}, vol. 27, no. 4, pp. 2597--2633,
  2017.

\bibitem{rokade2020distributed}
K.~Rokade and R.~K. Kalaimani,
\newblock ``Distributed {ADMM} over directed networks,''
\newblock {\em arXiv e-prints, arXiv:2010.10421}, 2020.

\bibitem{jiang2021asynchronous}
W.~Jiang, A.~Grammenos, E.~Kalyvianaki, and T.~Charalambous,
\newblock ``An asynchronous approximate distributed alternating direction
  method of multipliers in digraphs,''
\newblock in {\em IEEE Conference on Decision and Control (CDC)}, 2021, pp.
  3406--3413.

\bibitem{khatana2020dc}
V.~Khatana and M.~V. Salapaka,
\newblock ``{DC-DistADMM: ADMM} algorithm for contsrained distributed
  optimization over directed graphs,''
\newblock {\em arXiv preprint, arXiv:2003.13742}, 2020.

\bibitem{freris2010fundamental}
N.~M. Freris, S.~R. Graham, and P.~R. Kumar,
\newblock ``Fundamental limits on synchronizing clocks over networks,''
\newblock {\em IEEE Transactions on Automatic Control}, vol. 56, no. 6, pp.
  1352--1364, 2010.

\bibitem{makhdoumi2015graph}
A.~Makhdoumi and A.~Ozdaglar,
\newblock ``Graph balancing for distributed subgradient methods over directed
  graphs,''
\newblock in {\em IEEE Conference on Decision and Control (CDC)}, 2015, pp.
  1364--1371.

\end{thebibliography}
\section*{Appendix}
\textbf{\emph{Proof of Lemma~1}}: For each $k$, we consider the $l-$th entry of each  $\xi_i^k\left(\cdot\right)$ and combine them to form an $n-$dimensional vector denoted by $\widetilde{\xi}^k\left(\cdot\right)$, and we let $$\widetilde{\xi}^k_\perp\left(\cdot\right):=\left(I-\textbf{11}^T/n\right)\widetilde{\xi}^k\left(\cdot\right).$$ We define $W^k(b):=I-(D-A)\text{diag}\left(w^k(b)\right)$; this is column stochastic ($1^TW^k(b) = 1^T$), which shows that $\xi^{k}(b+1)$ and $\xi^{k}(b)$ have the same sum  (hence also average) since the update of the $l-$th entry can be written as:$$\widetilde{\xi}^k\left(b+1\right)=W^k(b)\widetilde{\xi}^k\left(b\right).$$The same reasoning also shows that $\bar{x}^k=\bar{z}^k$. Then, according to the first property of \cite[Lemma 2]{rokade2020distributed}, there exists $p^k(b)$, so that $W^k(b)p^k(b)=p^k(b)$ and $(p^k(b))^T\textbf{1}=1$. For any $b\in\{0,\hdots,B-1\}$, it follows from (\ref{updatez}) that 
\begin{align*}
    &\left(I-p^k(b)\textbf{1}^T\right)\left(\widetilde{\xi}^k(b+1)\right)\\=&\left(I-p^k(b)\textbf{1}^T\right)W^k(b)\widetilde{\xi}^k(b)\\=&\left(W^k(b)-p^k(b)\textbf{1}^T\right)\left(I-p^k(b)\textbf{1}^T\right)\widetilde{\xi}^k(b).
\end{align*}
Additionally,
\begin{align*}
&\left(I-p^k(b)\textbf{1}^T+\textbf{11}^T/n-\textbf{11}^T/n\right)\widetilde{\xi}^k(b+1)\\=&\left(W^k(b)-p^k(b)\textbf{1}^T\right)\left(I-p^k(b)\textbf{1}^T\right.\\&\left.+\textbf{11}^T/n-\textbf{11}^T/n\right)\widetilde{\xi}^k(b),
\end{align*}
which implies
\begin{align*}
    &\widetilde{\xi}^k_\perp(b+1)+\left(\textbf{1}/n-p^k(b)\right)\textbf{1}^T\widetilde{\xi}^k(b+1)\\=&\left(W^k(b)-p^k(b)\textbf{1}^T\right)\left[\widetilde{\xi}^k_\perp(b)+(\textbf{1}/n-p^k(b))\textbf{1}^T\widetilde{\xi}^k(b)\right].
\end{align*}
Then from the third property of \cite[Lemma 2]{rokade2020distributed}, for all $k \geq \bar{k}$
\begin{footnotesize}
\begin{equation*}
    \norm{\widetilde{\xi}^k_\perp(b+1)+(\textbf{1}/n-p^k(b))\textbf{1}^T\widetilde{x}^k} \leq\delta \norm{\widetilde{\xi}^k_\perp(b)+(\textbf{1}/n-p^k(b))\textbf{1}^T\widetilde{x}^k}.
\end{equation*}
\end{footnotesize}
Iterating the inequality over $b=0$ to $b=B-1$, we obtain
\begin{footnotesize}
\begin{equation*}
    \norm{\widetilde{z}^k_\perp+(\textbf{1}/n-p^k(B))\textbf{1}^T\widetilde{x}^k} \leq \delta^B\norm{\widetilde{x}^k_\perp+(\textbf{1}/n-p^k(0))\textbf{1}^T\widetilde{x}^k}.
\end{equation*}
\end{footnotesize}
Letting $\hat{p}^k = \arg\max_{p \in \{p^k(0),p^k(B)\}} \|1/n-p\|$, we get
\begin{align*}
    \norm{\widetilde{z}^k_\perp} &\leq \delta^B\norm{\widetilde{x}^k_\perp}+\left(1+\delta^B\right)\norm{(\textbf{1}/n-\hat{p}^k)\textbf{1}^T\widetilde{x}^k}\\&\leq \delta^B\norm{\widetilde{x}^k_\perp}+2\norm{(\textbf{1}/n-\hat{p}^k)\textbf{1}^T\widetilde{x}^k},
\end{align*}
which implies (using $\delta \in (0,1), 
(a+b)^2 \le \frac{1}{\delta} a^2 + \frac{1}{1-\delta}b^2$ for any $a,b\in\mathbb{R}$, 
along with Jensen's inequality): 
\begin{align*}
    &\norm{\widetilde{z}^k_\perp}^2 \leq\delta^{2B-1}\norm{\widetilde{x}^k_\perp}^2+\frac{4}{1-\delta}\norm{(\textbf{1}/n-\hat{p}^k)\textbf{1}^T\widetilde{x}^k}^2\\
    =&\delta^{2B-1}\norm{\widetilde{x}^k_\perp}^2\\&+\frac{4}{1-\delta}\norm{(\textbf{1}/n-\hat{p}^k)\left(\textbf{1}^T\widetilde{x}^k-\textbf{1}^T\widetilde{x}^\star+\textbf{1}^T\widetilde{x}^\star\right)}^2\\
    =&\delta^{2B-1}\norm{\widetilde{x}^k_\perp}^2\\&+\frac{4}{1-\delta}\norm{\textbf{1}/n-\hat{p}^k}^2\left(\textbf{1}^T\widetilde{x}^k-\textbf{1}^T\widetilde{x}^\star+\textbf{1}^T\widetilde{x}^\star\right)^2\\
    \leq&\delta^{2B-1}\norm{\widetilde{x}^k_\perp}^2\\&+\frac{8}{1-\delta}\norm{\textbf{1}/n-\hat{p}^k}^2\left(\left(\textbf{1}^T\widetilde{x}^k-\textbf{1}^T\widetilde{x}^\star\right)^2+\left(\textbf{1}^T\widetilde{x}^\star\right)^2\right)\\
    \leq&\delta^{2B-1}\norm{\widetilde{x}^k_\perp}^2\\&+\frac{8}{1-\delta}\norm{\textbf{1}/n-\hat{p}^k}^2\left(n\norm{\widetilde{x}^k-\widetilde{x}^\star}^2+n\norm{\widetilde{x}^\star}^2\right)
\end{align*}
Since $l$ is arbitrarily chosen from $\left\{1, \dots, d\right\}$, the inequality above holds for each entry position. Adding the $d$ inequalities together and defining $\Delta^k:=\frac{8n}{1-\delta}\norm{\textbf{1}/n-\hat{p}^k}^2$, which in view of the fact that both $\norm{\textbf{1}/n-p^k(0)}^2$ and $\norm{\textbf{1}/n-p^k(B)}^2$ tend to zero geometrically with the rate to be $\lambda_2(P)$ \cite{makhdoumi2015graph}, completes the proof.$\hfill\blacksquare$\\

\textbf{\emph{Proof of Theorem~1}}: For notational simplification, we denote $\nabla F\left(x^k\right)$ by $\nabla F^k$ and $\nabla F\left(x^\star\right)$ by $\nabla F^\star$. From (\ref{updatex}) and (\ref{KKTdual}) we have $$\nabla F^k-\nabla F^\star=-\frac{1}{\eta}(x^{k+1}-x^k)-(y^k-y^\star)-\rho(x^k-z^k)$$ and
\begin{align*}
&(x^k-x^\star)^T(\nabla F^k-\nabla F^\star)=\\&\underbrace{(x^{k+1}-x^\star)^T(\nabla F^k-\nabla F^\star)}_{(i)}+(x^k-x^{k+1})^T(\nabla F^k-\nabla F^\star)
\end{align*}
$(i)=\underbrace{-\frac{1}{\eta}(x^{k+1}-x^\star)^T(x^{k+1}-x^k)}_{(ii)}\\\underbrace{-(x^{k+1}-x^\star)^T(y^k-y^\star)}_{(iii)}\underbrace{-\rho(x^{k+1}-x^\star)^T(x^k-z^k)}_{(iv)}$.\\ 
The following is to bound the three terms.$$(ii)=\frac{1}{2\eta}\left(\norm{x^k-x^\star}^2-\norm{x^{k+1}-x^\star}^2-\norm{x^{k+1}-x^k}^2\right),$$where we use parallelogram law $$-2(a-b)^T(a-c)=\norm{b-c}^2-\norm{a-b}^2-\norm{a-c}^2,$$ which is also used several times in the sequel.
\begin{align*}
    &(iii)=-(x^{k+1}-z^{k+1}+z^{k+1}-x^\star)^T(y^k-y^\star)\\=&-(x^{k+1}-z^{k+1})^T(y^k-y^\star)-(z^{k+1}-x^\star)^T(y^k-y^\star)\\\underset{(\ref{updatea})}{=}&-\frac{1}{\rho}(y^{k+1}-y^k)^T(y^k-y^\star)-(z^{k+1}_\perp)^T(y^k-y^\star)\\=&-\frac{1}{2\rho}\left(\norm{y^{k+1}-y^\star}^2-\norm{y^k-y^{k+1}}^2-\norm{y^k-y^\star}^2\right)\\&-(z^{k+1}_\perp)^T(y^k-y^\star),
\end{align*}
where we use $z^{k+1}=\bar{x}^{k+1}+z^{k+1}_\perp$  and $\sum_{i=1}^n y_i^k=\sum_{i=1}^n y_i^\star=0$ from (\ref{KKTdual}).
\begin{align*}
    &(iv)=-\rho(x^{k+1}-x^\star)^T(x^k-x^{k+1}+x^{k+1}-z^k)\\=&-\rho(x^{k+1}-x^\star)^T(x^{k+1}-z^k)-\rho(x^{k+1}-x^\star)^T(x^k-x^{k+1})\\=&\frac{\rho}{2}\left(\norm{x^\star-z^k}^2-\norm{x^{k+1}-x^\star}^2-\norm{x^{k+1}-z^k}^2\right)\\&-\frac{\rho}{2}\left(\norm{x^\star-x^k}^2-\norm{x^{k+1}-x^\star}^2-\norm{x^k-x^{k+1}}^2\right)\\=&\frac{\rho}{2}\left(\norm{x^\star-z^k}^2+\norm{x^k-x^{k+1}}^2-\norm{x^\star-x^k}^2\right.\\&\left.-\norm{x^{k+1}-z^k}^2\right).
\end{align*}
 Invoking Lemma 1 and using $\delta\in (0,1)$, we obtain:
\begin{align*}
    &\norm{x^\star-z^k}^2-\norm{x^\star-x^k}^2=\norm{z^k}^2-\norm{x^k}^2\\=&\norm{\bar{x}^k+z^k_\perp}^2-\norm{\bar{x}^k+x^k_\perp}^2\\=&\norm{\bar{x}^k}^2+\norm{z^k_\perp}^2+2\bar{x}^{kT}z^k_\perp-\norm{\bar{x}^k}^2-\norm{x^k_\perp}^2-2\bar{x}^{kT}x^k_\perp\\=&\norm{z^k_\perp}^2-\norm{x^k_\perp}^2+2\bar{x}^{kT}\left(z^k_\perp-x^k_\perp\right)\\=&\norm{z^k_\perp}^2-\norm{x^k_\perp}^2\\\leq& \left(\delta^{2B-1}-1\right)\norm{x^k_\perp}^2+\Delta^k \left(\norm{x^k-x^\star}^2+\norm{x^\star}^2\right) \\\leq &\Delta^k \left(\norm{x^k-x^\star}^2+\norm{x^\star}^2\right).
\end{align*}
 Adding the three terms together gives:
 \begin{align*}
     &(i) \leq \frac{1}{2\eta}\left(\norm{x^k-x^\star}^2-\norm{x^{k+1}-x^\star}^2\right)\\&+\frac{1}{2\rho}\left(\norm{y^k-y^\star}^2-\norm{y^{k+1}-y^\star}^2\right)\\&+(\frac{\rho}{2}-\frac{1}{2\eta})\norm{x^{k+1}-x^k}^2+\frac{\rho}{2}\Delta^k\left(\norm{x^k-x^\star}^2+\norm{x^\star}^2\right)\\&\underbrace{+\frac{1}{2\rho}\norm{y^k-y^{k+1}}^2-\frac{\rho}{2}\norm{x^{k+1}-z^k}^2}_{(v)}\underbrace{\vphantom{\frac{1}{2\rho}}-(z^{k+1}_\perp)^T(y^k-y^\star)}_{(vi)}.
 \end{align*}
 We proceed to bound $(v)$ and $(vi)$.\\From (\ref{updatea}), $y^k-y^{k+1}=-\rho\left(x^{k+1}-z^{k+1}\right)$, so
 \begin{align*}
     &(v)=\frac{\rho}{2}\left(\norm{x^{k+1}-z^{k+1}}^2-\norm{x^{k+1}-z^k}^2\right)\\=&\frac{\rho}{2}\left(\norm{x^{k+1}-\bar{x}^{k+1}-z^{k+1}_\perp}^2-\norm{x^{k+1}-\bar{x}^{k}-z^{k}_\perp}^2\right)\\=&\frac{\rho}{2}\left[\norm{x^{k+1}-\bar{x}^{k+1}}^2+\norm{z^{k+1}_\perp}^2-\norm{x^{k+1}-\bar{x}^{k}}^2-\norm{z^{k}_\perp}^2\right.\\&\left.-2(z^{k+1}_\perp)^T(x^{k+1}-\bar{x}^{k+1})+2(z^{k}_\perp)^T(x^{k+1}-\bar{x}^{k})\right] \\\leq&\frac{\rho}{2}\left[\norm{z^{k+1}_\perp}^2-2(z^{k+1}_\perp)^T(x^{k+1}-\bar{x}^{k+1})\right.\\&\left.+2(z^{k}_\perp)^T(x^{k+1}-\bar{x}^{k+1}+\bar{x}^{k+1}-\bar{x}^{k})\right]\\\leq& \frac{\rho}{2}\left[\norm{z^{k+1}_\perp}^2+\norm{z^{k+1}_\perp}^2+\norm{x^{k+1}_\perp}^2+2\norm{z^{k}_\perp}^2\right.\\&\left.+\norm{\bar{x}^{k+1}-\bar{x}^{k}}^2+\norm{x^{k+1}_\perp}^2\right]\\\leq&\frac{\rho}{2}\left\{(2\delta^{2B-1}+2)\norm{x^{k+1}_\perp}^2+2\delta^{2B-1}\norm{x^k_\perp}^2\right.\\&\left.+\norm{\bar{x}^{k+1}-\bar{x}^k}^2+2\Delta^k(\norm{x^k-x^\star}^2+\norm{x^\star}^2)\right.\\&\left.+2\Delta^{k+1}(\norm{x^{k+1}-x^\star}^2+\norm{x^\star}^2)\right\}.
 \end{align*}
From (\ref{updatex}) and (\ref{KKTdual}),
$$y^k=-\frac{1}{\eta}(x^{k+1}-x^k)-\nabla F^k-\rho(x^k-z^k),$$
\begin{align*}
    &(vi)=-(z^{k+1}_\perp)^T\left[-\frac{1}{\eta}\left(x^{k+1}-x^k\right)-\left(\nabla F^k-\nabla F^\star\right)\right.\\&\left.\qquad \,\,\,\,-\rho \left(x^k-z^k\right)\right]\\\leq& \frac{\gamma}{4\eta}\left(\delta^{2B-1} \norm{x^{k+1}_\perp}^2+\Delta^{k+1}\left(\norm{x^{k+1}-x^\star}^2+\norm{x^\star}^2\right)\right)\\+&\frac{1}{\gamma \eta}\norm{x^{k+1}-x^k}^2+\frac{1}{\zeta}\norm{\nabla F^\star-\nabla F^k}^2+\rho\norm{x^k-z^k}^2\\+&\frac{\zeta}{4}\left(\delta^{2B-1} \norm{x^{k+1}_\perp}^2+\Delta^{k+1}\left(\norm{x^{k+1}-x^\star}^2+\norm{x^\star}^2\right)\right)\\+&\frac{\rho}{4}\left(\delta^{2B-1} \norm{x^{k+1}_\perp}^2+\Delta^{k+1}\left(\norm{x^{k+1}-x^\star}^2+\norm{x^\star}^2\right)\right),
\end{align*}
for arbitrary $\gamma,\zeta>0$. Note that
$\norm{\bar{x}^{k+1}-\bar{x}^k}^2 \leq \norm{x^{k+1}-x^k}^2$, and
\begin{align*}
    &\norm{x^k-z^k}^2=\norm{x^k_\perp-z^k_\perp}^2 \leq 2\norm{x^k_\perp}^2+2\norm{z^k_\perp}^2\\\leq&(2+2\delta^{2B-1})\norm{x^k_\perp}^2+2\Delta^k\left(\norm{x^k-x^\star}^2+\norm{x^\star}^2\right).
\end{align*}
For arbitrary $\tau>0$,
\begin{footnotesize}
\begin{align*}
    (\nabla F^k-\nabla F^\star)^T(x^k-x^{k+1})\leq \tau\norm{x^k-x^{k+1}}^2+\frac{1}{4\tau}\norm{\nabla F^k-\nabla F^\star}^2.
\end{align*}
\end{footnotesize}
Assumption \ref{aboutf} implies (cocoercivity of the gradient):
\begin{align*}
    &\left(x^k-x^{\star}\right)^T\left(\nabla F^k-\nabla F^{\star}\right) \\\geq& \frac{m_f M_f}{m_f+M_f}\norm{x^k-x^{\star}}^2+\frac{1}{m_f+M_f}\norm{\nabla F^k-\nabla F^{\star}}^2.
\end{align*}
Putting everything together we obtain that for some $\widetilde{k}$ and all $k \geq \widetilde{k}$:
\begin{align*}
    0 \leq& \left(\frac{1}{4\tau}+\frac{1}{\zeta}-\frac{1}{m_f+M_f}\right)\norm{\nabla F^\star-\nabla F^k}^2\\+&\frac{1}{2\rho}\norm{y^k-y^\star}^2+\left(\tau+\rho+\frac{1}{\gamma \eta}-\frac{1}{2\eta}\right)\norm{x^k-x^{k+1}}^2\\-&\frac{1}{2\rho}\norm{y^{k+1}-y^\star}^2\\+&\left\{\frac{1}{2\eta}+2\rho+3\rho\delta^{2B-1}-\frac{m_f M_f}{m_f+M_f}\right\}\norm{x^k-x^\star}^2\\-&\left\{\frac{1}{2\eta}-\delta^{2B-1}\left[\frac{5\rho}{4}+\frac{\gamma}{4\eta}+\frac{\zeta}{4}\right]-\rho\right\}\norm{x^{k+1}-x^\star}^2\\+&\frac{7\rho}{2}\Delta^k\left(\norm{x^k-x^\star}^2+\norm{x^\star}^2\right)\\+&\left(\frac{\gamma}{4\eta}+\frac{\zeta}{4}+\frac{5\rho}{4}\right)\Delta^{k+1}\left(\norm{x^{k+1}-x^\star}^2+\norm{x^\star}^2\right).
\end{align*}
By applying Jensen's inequality and (\ref{updatex}) along with the previously obtained bound on $\|x^k-z^k\|$,
\begin{align*}
    \frac{1}{3}\norm{y^k-y^\star}^2 \leq& \frac{1}{\eta^2}\norm{x^{k+1}-x^k}^2+\norm{\nabla F^\star-\nabla F^k}^2\\&+\rho^2(2+2\delta^{2B-1})\norm{x^k-x^\star}^2\\&+2\rho^2\Delta^k\left(\norm{x^k-x^\star}^2+\norm{x^\star}^2\right).
\end{align*}
It thus all boils down to choosing  suitable $B,\eta,\tau,\zeta$, such that the following inequalities hold:
\begin{subequations}
\begin{gather}
    \frac{1}{4\tau}+\frac{1}{\zeta}-\frac{1}{m_f+M_f} <0,\\
    \tau+\rho+\frac{1}{\gamma\eta}-\frac{1}{2\eta} <0,\\
    \frac{1}{2\eta}-\left[\delta^{2B-1}\left(\frac{5\rho}{4}+\frac{\gamma}{4\eta}+\frac{\zeta}{4}\right)+\rho\right]>0,\\
    \delta^{2B-1}\left(\frac{17\rho}{4}+\frac{\gamma}{4\eta}+\frac{\zeta}{4}\right)+3\rho-\frac{m_f M_f}{m_f+M_f}<0.
\end{gather}
\end{subequations}
By choosing $\gamma=4, \tau=\frac{3(M_f+m_f)}{4},\eta\leq\frac{4}{15(M_f+m_f)},\zeta=3(M_f+m_f)$, there exist  $\mu_1,\mu_2,\mu_3,k_1$, such that $0<\mu_1<\mu_2$ and for all $k>k_1$, the following holds:
\begin{align*}
    &\mu_1\left(\norm{x^k-x^\star}^2+\norm{y^k-y^\star}^2\right)+\mu_3\left(\Delta^k+\Delta^{k+1}\right)\norm{x^\star}^2\\&\geq \norm{x^{k+1}-x^\star}^2+\norm{y^{k+1}-y^\star}^2
\end{align*}
where $\mu_1=\max\left\{\frac{c_1+c_2}{2c_2},\left(\frac{1}{2\rho}-c_3\right)/\left(\frac{1}{2\rho}\right)\right\}$ and $c_1,c_2,c_3$ are specified in the statement of the theorem.\\Denoting $\omega^k:=\mu_3\left(\Delta^k+\Delta^{k+1}\right)\norm{x^\star}^2$, we have established
\begin{align*}
    &\mu_1\left(\norm{x^k-x^\star}^2+\norm{y^k-y^\star}^2\right)+\omega^k\\\geq&\norm{x^{k+1}-x^\star}^2+\norm{y^{k+1}-y^\star}^2,
\end{align*}
and $\omega^k$ tends to zero geometrically with rate $\lambda_2(P)$. We consider two cases:\\
$\bullet$ if there exists a subsequence $k_{(l)}$ such that
\begin{equation}
    \frac{1-\mu_1}{2}\left(\norm{x^{k_{(l)}+1}-x^\star}^2+\norm{y^{k_{(l)}+1}-y^\star}^2\right)<\omega^{k_{(l)}},\label{imp}
\end{equation}
then the subsequence $\norm{x^{k_{(l)}+1}-x^\star}^2+\norm{y^{k_{(l)}+1}-y^\star}^2 \rightarrow 0$ geometrically with rate $\lambda_2(P)$ and hence also $\lambda$. For $k$ that does not satisfy (\ref{imp}), we have
\begin{align}
    &\mu_1\left(\norm{x^k-x^\star}^2+\norm{y^k-y^\star}^2\right) \notag\\\geq&\frac{1+\mu_1}{2}\left(\norm{x^{k+1}-x^\star}^2+\norm{y^{k+1}-y^\star}^2\right),\label{14}
\end{align}
so again this shows asymptotic rate $\lambda$.\\
$\bullet$ if there exists some $k_2$ so that for all $k \geq k_2$, the reverse of (\ref{imp}) holds, then
 for all $k$, (\ref{14}) holds.\\
We conclude that in both cases, $\norm{x^k-x^\star}^2+\norm{y^k-y^\star}^2 \rightarrow 0$ geometrically with rate $\lambda$.$\hfill\blacksquare$\\

\textbf{\emph{Proof of Theorem~2}}: 
Let $\Omega^k \in \mathbb{R}^{n \times n}$ be a $0-1$ diagonal matrix where $1$ corresponds to the case agent $i$ is active and $0$ when it is inactive. Let $\mathbb{E}\left[\Omega^k\right]=\Omega:=\text{diag}(q_1,\hdots,q_n)$.\\The update of $w$ can be written compactly as $$w^k(b+1)=w^k(b)+\Omega^k\left(Pw^k(b)-w^k(b)\right).$$Let $w^\infty=\underset{t \rightarrow \infty}{\text{lim}}P^tw^0(0)$, then we have
\begin{align*}
    &\norm{w^k(b+1)-w^\infty}^2_{\Omega^{-1}}\\=&\norm{w^k(b)+\Omega^k\left(Pw^k(b)-w^k(b)\right)-w^\infty}^2_{\Omega^{-1}}\\=&\norm{w^k(b)-w^\infty}^2_{\Omega^{-1}}\\&+\left(Pw^k(b)-w^k(b)\right)^T\Omega^k\Omega^{-1}\Omega^k\left(Pw^k(b)-w^k(b)\right)\\&+2\left(w^k(b)-w^\infty\right)^T\Omega^{-1}\left(\Omega^k\left(Pw^k(b)-w^k(b)\right)\right)\\=&\norm{w^k(b)-w^\infty}^2_{\Omega^{-1}}\\&+\left(Pw^k(b)-w^k(b)\right)^T\Omega^{-1}\Omega^k\left(Pw^k(b)-w^k(b)\right)\\&+2\left(w^k(b)-w^\infty\right)^T\Omega^{-1}\left(\Omega^k\left(Pw^k(b)-w^k(b)\right)\right),
\end{align*}
 where we have used idempotency of $\Omega^k$ (as 0-1 diagonal matrix) and the fact that diagonal matrices commute. By taking conditional expectation (denoted by $\mathbb{E}^k$, where conditioning is on past activations) on both sides, we have 
 \begin{align*}
     &\mathbb{E}^k\left[\norm{w^k(b+1)-w^\infty}^2_{\Omega^{-1}}\right]\\=&\norm{w^k(b)-w^\infty}^2_{\Omega^{-1}}+\norm{Pw^k(b)-w^k(b)}^2\\&+2\left(w^k(b)-w^\infty\right)^T\left(Pw^k(b)-w^k(b)\right).
 \end{align*}
  From \cite[Lemma 1]{makhdoumi2015graph}, there exists some positive $\theta<1$ only depending on $P$, such that
  \begin{align*}
      &\norm{Pw^k(b)-w^k(b)}^2+2\left(w^k(b)-w^\infty\right)^T\left(Pw^k(b)-w^k(b)\right)\\&\leq -\theta\norm{w^k(b)-w^\infty}^2,
  \end{align*}
  which implies
  \begin{footnotesize}
\begin{align*}
    \mathbb{E}^k\left[\norm{w^k(b+1)-w^\infty}^2_{\Omega^{-1}}\right] \leq\left(1-q_{\text{min}}\theta\right) \norm{w^k(b)-w^\infty}^2_{\Omega^{-1}}.
\end{align*}
\end{footnotesize}
   By induction, we obtain $$\mathbb{E}\left[\norm{w^k(b)-w^\infty}^2\right]\leq c_1 \lambda^{kB+b}_2$$for some constant $c_1$ and $\lambda_2:=1-q_{\text{min}}\theta$,  whence invoking 
 Markov's inequality for arbitrary $\epsilon_1 >0$ gives $$\text{Pr}\left[\norm{w^k(b)-w^\infty}^2 \geq \epsilon_1\right] \leq \frac{1}{\epsilon_1}c_1\lambda_2^{kB+b}.$$ This further implies $$\underset{n \rightarrow \infty}{\text{lim}}\text{Pr}\left[\cup^\infty_{k=n}\cup^{B-1}_{b=0}\left\{\norm{w^k(b)-w^\infty}^2>\epsilon_1\right\}\right]=0,$$ i.e., $\forall \epsilon>0$, $\exists K_1(\epsilon)$, such that $$\text{Pr}\left[\cup^\infty_{k=K_1}\cup^{B-1}_{b=0}\left\{\norm{w^k(b)-w^\infty}^2>\epsilon_1\right\}\right]\leq \frac{\epsilon}{2},$$ which implies
 \begin{align}
    \text{Pr}\left[\cap^\infty_{k=K_1}\cap^{B-1}_{b=0}\left\{\norm{w^k(b)-w^\infty}^2<\epsilon_1\right\}\right]\geq 1-\frac{\epsilon}{2}.\label{possibility}
 \end{align}
 We define $v^k:=\left[(x^k)^T,(y^k)^T\right]^T$, and let $T_a$ to be the operator corresponding to Alg.~1 for full participation, i.e., $T_a: (x^k,y^k,z^k) \mapsto (x^{k+1},y^{k+1},z^{k
+1})$ and $T:=\begin{bmatrix}I_d\quad 0\quad 0\\0 \,\quad I_d \quad 0\end{bmatrix}T_a$.\\We define $\Phi^k:=\text{blkdiag}\left(\Omega^k \otimes I_d,\Omega^k \otimes I_d\right)$ and $\Phi^k_a:=\text{blkdiag}\left(\Omega^k \otimes I_d,\Omega^k \otimes I_d,\Omega^k \otimes I_d\right)$, then \\$\mathbb{E}\left[\Phi^k\right]=\Phi:=\text{blkdiag}\left(\Omega \otimes I_d,\Omega \otimes I_d\right)$ and we have \\
\centerline{$\Phi^k T=\Phi^k \begin{bmatrix}I_d\quad 0\quad 0\\0 \,\quad I_d \quad 0\end{bmatrix}T_a=\begin{bmatrix}I_d\quad 0\quad 0\\0 \,\quad I_d \quad 0\end{bmatrix} \Phi_a^k T_a$,} which allows the analysis to carry for $T$ only.\\
For partial participation, we have $$v^{k+1}=v^k+\Phi^k\left(Tv^k-v^k\right),$$ i.e., $$\norm{v^{k+1}-v^\star}^2_{\Phi^{-1}}=\norm{v^k-v^\star+\Phi^k\left(Tv^k-v^k\right)}^2_{\Phi^{-1}}.$$ There are two  possibilities (depending on whether (\ref{imp}) is satisfied) either $$\norm{Tv^k-v^\star}^2\leq \frac{2\mu_1}{\mu_1+1}\norm{v^k-v^\star}^2$$ or $$\norm{Tv^k-v^\star}^2\leq \frac{2}{1-\mu_1}\omega^k,$$where $\omega^k \rightarrow 0$. For the first case we have 
\begin{footnotesize}
\begin{align*}
    \mathbb{E}^k\left[\norm{v^{k+1}-v^\star}^2_{\Phi^{-1}}\right] \leq \left(1-(1-\frac{2\mu_1}{\mu_1+1})q_{\min}\right) \norm{v^{k}-v^\star}^2_{\Phi^{-1}}.
\end{align*}
\end{footnotesize}
For the second we have
\begin{align*}
    &\norm{v^{k+1}-v^\star}^2_{\Phi^{-1}}=\norm{v^k-v^\star+\Phi^k\left(Tv^k-v^k\right)}^2_{\Phi^{-1}}\\=&\norm{\Phi^k(Tv^k-v^\star)+(I-\Phi^k)(v^k-v^\star)}^2_{\Phi^{-1}}\\\leq& \frac{1}{\beta}\norm{\Phi^k(Tv^k-v^\star)}^2_{\Phi^{-1}}+\frac{1}{1-\beta}\norm{(I-\Phi^k)(v^k-v^\star)}^2_{\Phi^{-1}},
\end{align*}
 taking expectation and noting that$$\mathbb{E}\left[\norm{(I-\Phi^k)(v^k-v^\star)}^2_{\Phi^{-1}}\right]\leq (1-q_{\min})^2\norm{v^k-v^\star}^2_{\Phi^{-1}},$$ implies that$$\mathbb{E}^k\left[\norm{v^{k+1}-v^\star}^2_{\Phi^{-1}}\right]\leq \frac{1}{\beta}\omega^k+\frac{(1-q_{\min})^2}{1-\beta}\norm{v^k-v^\star}^2_{\Phi^{-1}},$$ for any $\beta>0$, whence choosing $\beta=q_{\min}$, we obtain$$\mathbb{E}^k\left[\norm{v^{k+1}-v^\star}^2_{\Phi^{-1}}\right]\leq \frac{1}{q_{\min}}\omega^k+(1-q_{\min})\norm{v^k-v^\star}^2_{\Phi^{-1}}.$$
 By induction, we conclude that for both cases
 $$\mathbb{E}\left[\norm{v^{k+1}-v^\star}^2_{\Phi^{-1}}\right] \leq c_2 \lambda_1^k,$$
 where $\lambda_1$ is as Thm.~2 states.\\
An analogous use of Markov's inequality in the establishment of (\ref{possibility}) concludes there exists some $K_2$, such that$$\text{Pr}\left[\cap_{k=K_2}^\infty \norm{v^{k+1}-v^\star}^2_{\Phi^{-1}} \leq \lambda_p^k\right] \geq 1-\frac{\epsilon}{2},$$
where $\lambda_p$ can be arbitrarily chosen from $\left(\lambda_1,1\right)$, and hence finishes the proof.$\hfill\blacksquare$

\end{document}